\documentclass[letter]{article}

\usepackage{geometry}

\usepackage{expl3}

\usepackage{amssymb, amsthm, mathtools}
\usepackage[unicode,colorlinks=true,linkcolor=blue,urlcolor=magenta, citecolor=blue]{hyperref}

\usepackage[warnings-off={mathtools-colon,mathtools-overbracket}]{unicode-math}

\newtheorem{theorem}{Theorem}

\newtheorem{lemma}{Lemma}
\newtheorem{corollary}{Corollary}
\theoremstyle{definition}
\newtheorem{definition}{Definition}

\newcommand{\lean}[1]{}
\newcommand{\discussion}[1]{}
\newcommand{\leanok}{}

\ExplSyntaxOn
\NewDocumentCommand{\uses}{m}
 {\clist_map_inline:nn{#1}{\vphantom{\ref{##1}}}\ignorespaces}
\NewDocumentCommand{\proves}{m}
 {\clist_map_inline:nn{#1}{\vphantom{\ref{##1}}}\ignorespaces}
\ExplSyntaxOff

\usepackage{cleveref}

\crefname{definition}{Definition}{Definitions}
\Crefname{definition}{Definition}{Definitions}
\crefname{lemma}{Lemma}{Lemmas}
\Crefname{lemma}{Lemma}{Lemmas}
\crefname{corollary}{Corollary}{Corollaries}
\Crefname{corollary}{Corollary}{Corollaries}

\title{Embedding Finite Functions into Low-Degree Polynomial Functions over Commutative Rings}
\author{
  Roman Bacik \\
  \small Vancouver, Canada \\
}
\date{\today}

\begin{document}
\maketitle
% Blueprint content for the FinBin project.
% This file is shared by the web and print versions; it must not contain
% \begin{document}. All statements carry blueprint labels (\label, \uses, \lean,
% \leanok) binding them to the machine-checked Lean declarations in the `Finbin`
% library.

\begin{abstract}
A function $f \colon X^k \to X$ on a finite set embeds into a polynomial of total degree $d$
over a commutative ring $R$ if there is an injection $j \colon X \to R$ and a polynomial
$g$ of total degree at most $d$ with $j \circ f = g \circ j^k$, where $j^k$ applies $j$ in
each coordinate. These are the
transition functions of $k$-neighbour cellular automata, and the injection $j$ is an
enlargement of the alphabet that preserves the transitions. We prove three results, all
verified in Lean~4 with Mathlib~\cite{bacik2026finbin}. Every unary function $f \colon X \to X$ embeds into a
polynomial of total degree $1$. Every binary Kronecker delta embeds into a polynomial of
total degree $4$. For every $d$ there is a binary function that does not embed into any
polynomial of total degree $d$.
\end{abstract}

\section{Introduction}

A cellular automaton evolves a configuration over an alphabet $X$ by applying, at every
cell at once, a fixed local rule that reads $k$ neighbouring cells. The rule is a transition
function $f \colon X^k \to X$, and the automaton is determined by $f$ and its neighbourhood.
Smith~\cite{Smith1971} proved that even two-neighbour cellular automata can effectively
emulate Turing-machine computation.

We represent such a function algebraically. Fix an injection $j \colon X \to R$ into a
commutative ring $R$ and require that $f$ become the restriction along $j$ of a polynomial
map $R^k \to R$, i.e.\ that the square of \Cref{def:embed} commute. The injection relabels
the alphabet by elements of $R$; this is an enlargement of the alphabet, which changes
neither the transitions nor the computational power of the automaton.

Over a finite field every function is a polynomial function, so every $f$ is represented by
some polynomial. For $k = 1$ every $f \colon X \to X$ is represented in degree $1$
(\Cref{thm:unary}). For $k = 2$ every Kronecker delta is represented in degree $4$
(\Cref{thm:quartic}), while for every $d$ there is a binary function represented by no
polynomial of total degree $d$ (\Cref{cor:arbitrary}); the inversion indicator on
$\mathbb{Z}/p$ for a prime $p > d$ is such a function (\Cref{thm:obstruction}), and the
lower bound holds over every commutative ring, not only over $\mathbb{Z}/p$.

\Cref{sec:unary,sec:quartic} give the two constructions and \Cref{sec:obstruction} the lower
bound, after \Cref{sec:prelim} fixes notation.

\section{Preliminaries}\label{sec:prelim}

Throughout, $R$ denotes a commutative ring and $\mathbb{Z}/n$ the ring of integers modulo
$n$. We write $R[X_1, \dots, X_k]$ for the polynomial ring in $k$ variables and $\deg g$ for
the total degree of $g$, and identify each $g \in R[X_1,\dots,X_k]$ with the function
$R^k \to R$, $(r_1,\dots,r_k) \mapsto g(r_1,\dots,r_k)$, that it induces. For a map
$j \colon X \to R$ we write $j^k \colon X^k \to R^k$ for its coordinatewise extension
$j^k(x_1,\dots,x_k) = (j x_1,\dots,j x_k)$.

\begin{definition}[Polynomial embedding]\label{def:embed}\lean{Finbin.EmbedsInDegree}\leanok
Let $X$ be finite and $f \colon X^k \to X$. We say $f$ \emph{embeds in degree $d$} if there
exist a commutative ring $R$, an injection $j \colon X \to R$, and a polynomial
$g \in R[X_1,\dots,X_k]$ with $\deg g \le d$ such that $j \circ f = g \circ j^k$, i.e.\
the square
\[
  \begin{array}{ccc}
    X^k & \xrightarrow{\;f\;} & X \\[2pt]
    {\scriptstyle j^k}\big\downarrow & & \big\downarrow{\scriptstyle j} \\[2pt]
    R^k & \xrightarrow{\;g\;} & R
  \end{array}
\]
commutes.
\end{definition}

\begin{lemma}[Degree monotonicity]\label{lem:mono}
\lean{Finbin.EmbedsInDegree.mono}\leanok\uses{def:embed}
If $f$ embeds in degree $d$ and $d \le d'$, then $f$ embeds in degree $d'$.
\end{lemma}

\begin{proof}\leanok
The same $R$, $j$, and $g$ satisfy $\deg g \le d \le d'$.
\end{proof}

\begin{definition}[Kronecker delta]\label{def:kron}\lean{Finbin.kroneckerDelta}\leanok
For $a, b \in \mathbb{Z}/n$ the binary Kronecker delta
$\delta_{a,b} \colon (\mathbb{Z}/n)^2 \to \mathbb{Z}/n$ is
\[
  \delta_{a,b}(x,y) = \begin{cases} 1 & (x,y) = (a,b), \\ 0 & \text{otherwise.}\end{cases}
\]
More generally, for a point $a = (a_1,\dots,a_k)$ the $k$-ary delta $\delta_a$ takes the
value $1$ at $a$ and $0$ at every other point.
\end{definition}

\begin{definition}[Inversion indicator]\label{def:inv}\lean{Finbin.invIndicator}\leanok
For a natural number $p$ the inversion indicator
$\iota_p \colon (\mathbb{Z}/p)^2 \to \mathbb{Z}/p$ is $\iota_p(x,y) = 1$ if $xy = 1$ and $0$
otherwise.
\end{definition}

\section{Linear embedding of unary functions}\label{sec:unary}

\begin{theorem}[Linear representation]\label{thm:unary-lin}
\lean{Finbin.linear_representation}\leanok
For every $f \colon \mathbb{Z}/n \to \mathbb{Z}/n$ there are $m > 0$, $a \in \mathbb{Z}/m$
and an injection $j \colon \mathbb{Z}/n \to \mathbb{Z}/m$ with $j(f(x)) = a\,j(x)$ for all
$x$.
\end{theorem}

\begin{proof}\leanok
Set $N = n$. Iterating $f$ partitions $\mathbb{Z}/n$ into components, each a cycle with trees
feeding into it. For $x \in \mathbb{Z}/n$ let $\mathrm{depth}(x)$ be the least $d$ with
$f^{[d]}(x)$ on a cycle, $\mathrm{root}(x) = f^{[\mathrm{depth}(x)]}(x)$ the cycle point it
reaches, and $\pi(x)$ the length of that cycle; since $\mathbb{Z}/n$ has $n$ elements,
$\mathrm{depth}(x) \le n = N$. Each component has a distinguished cycle point $c$ (say the
least), and for a cycle point $z$ let $\mathrm{pos}(z) \in \{0,\dots,\pi-1\}$ be the number of
$f$-steps from $c$ to $z$. Let $\mathrm{sig}(x)$ be the position of $x$ in the list of
$f$-preimages of $f(x)$; then $\mathrm{sig}(x) < n$, and
\begin{equation}\label{eq:sig}
  f(x) = f(y) \ \text{and}\ \mathrm{sig}(x) = \mathrm{sig}(y) \implies x = y. \tag{$\ast$}
\end{equation}

Fix one component, with cycle length $\pi$ and distinguished point $c$. By Dirichlet's theorem
choose distinct primes $p, q > n$ with $p \equiv 1 \pmod{\pi}$; since $\pi \mid p-1$, the group
$(\mathbb{Z}/p)^\times$ contains a primitive $\pi$-th root of unity $\zeta$. Define, for $x$ in
the component,
\[
  a(x) = \zeta^{\,\mathrm{pos}(\mathrm{root}(x)) - \mathrm{depth}(x)} \in \mathbb{Z}/p,
  \qquad
  b(x) = \sum_{k=0}^{\mathrm{depth}(x)-1} \bigl(\mathrm{sig}(f^{[k]}(x)) + 1\bigr)\,
         q^{\,N-1-k} \in \mathbb{Z}/q^{N}.
\]
As $0 \le \mathrm{sig}(\cdot)+1 \le n < q$, the integer $b(x)$ has base-$q$ digits less than
$q$, and $b(x) < q^{N}$.

\emph{Linearity.} Applying $f$ stays in the component. If $x$ is off the cycle then
$\mathrm{root}(f x) = \mathrm{root}(x)$ and $\mathrm{depth}(f x) = \mathrm{depth}(x) - 1$, so
$a(f x) = \zeta^{\,\mathrm{pos}(\mathrm{root} x) - \mathrm{depth}(x) + 1} = \zeta\, a(x)$. If
$x$ is on the cycle then $\mathrm{depth}(x) = 0$, $\mathrm{root}(x) = x$, and
$\mathrm{pos}(\mathrm{root}(f x)) \equiv \mathrm{pos}(\mathrm{root}(x)) + 1 \pmod{\pi}$; as
$\zeta^{\pi} = 1$, again $a(f x) = \zeta\, a(x)$. For $b$, reindexing the sum gives
$q\, b(x) = (\mathrm{sig}(x)+1)\,q^{N} + b(f x)$, hence $b(f x) = q\, b(x)$ in
$\mathbb{Z}/q^{N}$.

\emph{Combination.} Since $\gcd(p, q^{N}) = 1$, the Chinese remainder theorem is a ring
isomorphism $\mathbb{Z}/(p\,q^{N}) \cong \mathbb{Z}/p \times \mathbb{Z}/q^{N}$. Let $j_0(x)$
correspond to $(a(x), b(x))$ and $\mu$ to $(\zeta, q)$. By linearity
$j_0(f x) = \mu\, j_0(x)$, and $a(x)$ is a power of $\zeta$, hence a unit, so $j_0(x) \ne 0$.

\emph{Injectivity on the component.} Suppose $j_0(x) = j_0(y)$, i.e.\ $a(x) = a(y)$ and
$b(x) = b(y)$. The least nonzero base-$q$ digit of $b(x)$ is at position
$N - \mathrm{depth}(x)$, so $\mathrm{depth}(x) = \mathrm{depth}(y) =: d$, and equality of the
digits gives $\mathrm{sig}(f^{[k]}x) = \mathrm{sig}(f^{[k]}y)$ for $0 \le k < d$. From
$a(x) = a(y)$ and equal depth, $\zeta^{\mathrm{pos}(\mathrm{root} x)} =
\zeta^{\mathrm{pos}(\mathrm{root} y)}$; as $\zeta$ has order $\pi$ and $\mathrm{pos} < \pi$,
$\mathrm{pos}(\mathrm{root} x) = \mathrm{pos}(\mathrm{root} y)$, so within this one component
$\mathrm{root}(x) = \mathrm{root}(y)$. Then $f^{[d]}x = \mathrm{root}(x) = \mathrm{root}(y) =
f^{[d]}y$, and downward induction on $k = d-1,\dots,0$, using $f^{[k+1]}x = f^{[k+1]}y$ with
$\mathrm{sig}(f^{[k]}x) = \mathrm{sig}(f^{[k]}y)$ in \eqref{eq:sig}, gives $f^{[k]}x =
f^{[k]}y$; at $k = 0$, $x = y$.

\emph{Globalisation.} Carry this out for every component, choosing the primes so that
$r \mapsto p_r$ and $r \mapsto q_r$ are injective and $\{p_r\} \cap \{q_s\} = \varnothing$;
then the moduli $p_r q_r^{N}$ are pairwise coprime, and with $m = \prod_r p_r q_r^{N}$ the
Chinese remainder theorem gives $\mathbb{Z}/m \cong \prod_r \mathbb{Z}/(p_r q_r^{N})$. Let
$j(x)$ be the element whose coordinate at the component $r$ of $x$ is $j_0(x)$ and whose other
coordinates are $0$, and let $a$ be the element with coordinate $\mu_r$ at every $r$. Since $f$
preserves components, $j(f x) = a\, j(x)$ coordinatewise. If $j(x) = j(y)$, the unique nonzero
coordinate (nonzero because $j_0 \ne 0$) identifies the common component, and injectivity on
it gives $x = y$. Thus $m > 0$, $a \in \mathbb{Z}/m$, and $j \colon \mathbb{Z}/n \to
\mathbb{Z}/m$ is injective with $j(f(x)) = a\, j(x)$ for all $x$.
\end{proof}

An alternative proof, using the adjugate of the characteristic matrix of the functional
graph of $f$, is given in \cite{bacik2025linear}.

\begin{theorem}[Unary linear embedding]\label{thm:unary}
\lean{Finbin.embedsInDegree_unary}\leanok\uses{def:embed,thm:unary-lin}
Every function $f \colon \mathbb{Z}/n \to \mathbb{Z}/n$ embeds in degree $1$.
\end{theorem}

\begin{proof}\leanok\uses{thm:unary-lin}
Take $m$, $a$, $j$ as in \Cref{thm:unary-lin}, $R = \mathbb{Z}/m$, and $g = a\,X_1$, so
$\deg g \le 1$ and $g(r) = a\,r$. Then $g(j(x)) = a\,j(x) = j(f(x))$, i.e.\
$j \circ f = g \circ j^1$.
\end{proof}

\section{Degree-four embedding of binary Kronecker deltas}\label{sec:quartic}

\begin{theorem}[Quartic representation]\label{thm:quartic-rep}
\lean{Finbin.quartic_d}\leanok\uses{def:kron}
For every $n \ge 1$ and all $a, b \in \mathbb{Z}/n$ there exist a modulus $m$, an injection
$j \colon \mathbb{Z}/n \to \mathbb{Z}/m$, and $g \in (\mathbb{Z}/m)[X_1,X_2]$ of total
degree at most $4$ with $j(\delta_{a,b}(x,y)) = g(j(x),\, j(y))$ for all $x,y$.
\end{theorem}

\begin{proof}\leanok
Throughout, $\bar x \in \{0,\dots,n-1\}$ denotes the representative of $x \in \mathbb{Z}/n$.

\emph{Three degree-preserving reductions.} Let $(j,g)$ over $R = \mathbb{Z}/m$ represent
$\delta_{a,b}$.
\begin{enumerate}
\item[(i)] $(j,\, g(Y,X))$ represents $\delta_{b,a}$.
\item[(ii)] If $\sigma$ is a permutation of $\mathbb{Z}/n$ with $\sigma 0 = 0$ and
$\sigma 1 = 1$, then $(j \circ \sigma^{-1},\, g)$ represents $\delta_{\sigma a, \sigma b}$:
indeed $\delta_{\sigma a,\sigma b}(x,y) = \delta_{a,b}(\sigma^{-1}x, \sigma^{-1}y)$, and since
$\sigma$ fixes the outputs $0,1$, $(j\circ\sigma^{-1})(\delta_{\sigma a,\sigma b}(x,y)) =
j(\delta_{a,b}(\sigma^{-1}x,\sigma^{-1}y)) = g((j\circ\sigma^{-1})x,(j\circ\sigma^{-1})y)$.
\item[(iii)] For the transposition $\tau = (0\,1)$, the pair
$(j \circ \tau,\ (j 0 + j 1) - g)$ represents $\delta_{\tau a, \tau b}$: here
$\delta_{a,b}(\tau x,\tau y) = \delta_{\tau a,\tau b}(x,y)$ and $\tau$ swaps the outputs
$0,1$, so $(j 0 + j 1) - g(j\tau x, j\tau y) = (j 0 + j 1) - j(\delta_{\tau a,\tau b}(x,y))
= j(\tau(\delta_{\tau a,\tau b}(x,y))) = (j\circ\tau)(\delta_{\tau a,\tau b}(x,y))$.
\end{enumerate}
The polynomial keeps its total degree in each case.

\emph{Reduction to base forms.} A permutation fixing $0$ and $1$ sends any element of
$\mathbb{Z}/n \setminus \{0,1\}$ to any other. Hence by (i)--(iii) every $\delta_{a,b}$
reduces, in the same degree, to one of the following. If $a = b$: to $\delta_{0,0}$ or
$\delta_{2,2}$ (carry $a$ to $0$; to $1$, then apply (iii); or to $2$). If $a \ne b$: to
$\delta_{0,1}$ when $\{a,b\} = \{0,1\}$; to $\delta_{1,2}$, and to $\delta_{0,2}$ from it by
(iii), when exactly one of $a,b$ lies in $\{0,1\}$; and to $\delta_{2,3}$ when
$a,b \notin \{0,1\}$. The representatives $2,3$ need $n \ge 3$ and $n \ge 4$; for the finitely
many smaller $n$ each delta is represented directly.

\emph{Base constructions.} In each case below $j$ is injective and $\deg g \le 4$; one checks
$g(j x, j y) = j(\delta(x,y))$ by evaluating at the marked point and noting that elsewhere the
product is a multiple of the modulus, hence $0$ in $R$.
\begin{itemize}
\item $\delta_{0,0}$, degree $2$: $R = \mathbb{Z}/(n(n+1))$, $j(0) = 0$ and
$j(x) = \bar x(n+1) - 1$ for $x \ne 0$, $g = n\,(X+1)(Y+1)$; then $g(j0,j0) = n = j(1)$.
\item $\delta_{2,2}$, degree $2$ ($n \ge 3$): $R = \mathbb{Z}/(4n)$, $j(2) = 1$, $j(1) = 2n$,
$j(x) = 2\bar x$ otherwise, $g = 2n\,XY$; then $g(j2,j2) = 2n = j(1)$.
\item $\delta_{1,2}$, degree $3$: with $m = n$ if $n$ is odd and $m = n+1$ otherwise (so $m$
is odd), $R = \mathbb{Z}/(4m)$, $j(1) = 2m$, $j(2) = 3$, $j(x) = 4\bar x$ otherwise,
$g = m\,X(X+1)Y$; then $g(j1,j2) = 6m^2(2m+1) \equiv 2m = j(1) \pmod{4m}$, using that $m$ is
odd.
\item $\delta_{2,3}$, degree $3$ ($n \ge 4$): $R = \mathbb{Z}/(16n)$, $j(2) = 1$, $j(3) = 2$,
$j(1) = 8n$, $j(x) = 4\bar x$ otherwise, $g = 4n\,XY(Y-1)$; then $g(j2,j3) = 8n = j(1)$.
\item $\delta_{0,1}$, degree $4$: with $k = n^3 + n^2 + 1$, $R = \mathbb{Z}/(nk)$, $j(0) = 1$,
$j(1) = 1+n$, $j(x) = k\bar x$ otherwise, $g = 1 + n\,XY\,(X-(1+n))(Y-1)$. The second term
vanishes off $(0,1)$ — $Y-1$ kills $y = 0$, $X-(1+n)$ kills $x = 1$, and if $x \notin\{0,1\}$
or $y \notin\{0,1\}$ then $nX$ or $nY$ is a multiple of $nk$ — and at $(0,1)$ it equals
$n\cdot 1\cdot(1+n)(-n)\,n = -(n^4+n^3) = n - nk \equiv n$, so $g(j0,j1) = 1+n = j(1)$.
\end{itemize}
With the reductions, every $\delta_{a,b}$ is represented in degree at most $4$.
\end{proof}

\begin{theorem}[Quartic delta embedding]\label{thm:quartic}
\lean{Finbin.embedsInDegree_kronecker}\leanok\uses{def:embed,thm:quartic-rep}
For every $n \ge 1$ and all $a, b \in \mathbb{Z}/n$, the binary Kronecker delta
$\delta_{a,b}$ embeds in degree $4$.
\end{theorem}

\begin{proof}\leanok\uses{thm:quartic-rep}
Take $m$, $j$, and $g \in (\mathbb{Z}/m)[X_1,X_2]$ as in \Cref{thm:quartic-rep}, with
$R = \mathbb{Z}/m$. Then $\deg g \le 4$ and $g(j(x), j(y)) = j(\delta_{a,b}(x,y))$ for all
$x,y$, i.e.\ $j \circ \delta_{a,b} = g \circ j^2$.
\end{proof}

\section{No low-degree embedding of binary functions}\label{sec:obstruction}

\begin{lemma}[Nilpotent cancellation]\label{lem:nilpotent}
\lean{Finbin.nilpotent_cancel}\leanok
Let $R$ be a commutative ring. If $t \in R$ is nilpotent and $z\,(1+t) = 0$, then $z = 0$.
\end{lemma}

\begin{proof}\leanok
As $t$ is nilpotent, $t^{k} = 0$ for some $k \ge 1$. Put $u = \sum_{i=0}^{k-1} (-t)^{i}$.
Then $t\,(-t)^{i} = -(-t)^{i+1}$, so the sum telescopes:
\[
  (1+t)\,u = u + t\,u
           = \sum_{i=0}^{k-1} (-t)^{i} - \sum_{i=0}^{k-1} (-t)^{i+1}
           = (-t)^{0} - (-t)^{k} = 1 - (-1)^{k} t^{k} = 1 .
\]
Hence $1+t$ is a unit with inverse $u$. Multiplying $z\,(1+t) = 0$ on the right by $u$ gives
$z = z\,(1+t)\,u = 0$.
\end{proof}

\begin{lemma}[Algebraic core]\label{lem:core}
\lean{Finbin.inv_graph_core}\leanok\uses{lem:nilpotent}
Let $R$ be a commutative ring and $c, e \in R$. If $c^2 e = c$ and $c$ is nilpotent, then
$c = 0$.
\end{lemma}

\begin{proof}\leanok\uses{lem:nilpotent}
From $c^2 e = c$ we get $c\,(1 - ce) = 0$, i.e.\ $c\,(1 + t) = 0$ with $t = -ce$. As $c$ is
nilpotent so is $t$, and \Cref{lem:nilpotent} gives $c = 0$.
\end{proof}

\begin{lemma}[Rank-deficient determinant]\label{lem:rank}
\lean{Finbin.det_mul_eq_zero_of_card_lt}\leanok
If $A$ is an $n \times m$ matrix and $B$ an $m \times n$ matrix over a commutative ring with
$m < n$, then $\det(AB) = 0$.
\end{lemma}

\begin{proof}\leanok
Write the index sets as $[n]$ and $[m]$. By the Leibniz formula,
\[
  \det(AB) = \sum_{\sigma \in S_n} \operatorname{sgn}(\sigma) \prod_{i \in [n]} (AB)_{i,\sigma(i)}
           = \sum_{\sigma \in S_n} \operatorname{sgn}(\sigma) \prod_{i \in [n]}
             \sum_{l \in [m]} A_{i,l}\, B_{l,\sigma(i)} .
\]
Expanding the product over $i$ as a sum over functions $\varphi \colon [n] \to [m]$ and
exchanging the order of summation,
\[
  \det(AB) = \sum_{\varphi \colon [n] \to [m]} \Bigl(\prod_{i} A_{i,\varphi(i)}\Bigr)
             \sum_{\sigma \in S_n} \operatorname{sgn}(\sigma) \prod_{i} B_{\varphi(i),\sigma(i)}
           = \sum_{\varphi \colon [n] \to [m]} \Bigl(\prod_{i} A_{i,\varphi(i)}\Bigr)
             \det\bigl(B_{\varphi(i),k}\bigr)_{i,k},
\]
the inner sum being the Leibniz formula for the determinant of the $n \times n$ matrix with
$(i,k)$ entry $B_{\varphi(i),k}$. Because $m < n$, every $\varphi \colon [n] \to [m]$ is
non-injective: there are $i \ne i'$ with $\varphi(i) = \varphi(i')$, so that matrix has two
equal rows and its determinant is $0$. Hence every summand vanishes and $\det(AB) = 0$.
\end{proof}

\begin{theorem}[Inversion-graph obstruction]\label{thm:invp}
\lean{Finbin.thm_invp}\leanok\uses{def:inv,lem:core,lem:rank}
Let $p$ be prime, $R$ a commutative ring, $j \colon \mathbb{Z}/p \to R$ injective, and
$q \in R[X_1,X_2]$ with $\deg q \le p-1$. Then $q(j(x), j(y)) = j(\iota_p(x,y))$ cannot
hold for all $x, y \in \mathbb{Z}/p$.
\end{theorem}

\begin{proof}\leanok\uses{lem:core,lem:rank}
Suppose $q(j(x),j(y)) = j(\iota_p(x,y))$ for all $x,y \in \mathbb{Z}/p$, and write
$q = \sum_{s,t \ge 0} c_{s,t}\, X_1^{s} X_2^{t}$ with $c_{s,t} \in R$ and $c_{s,t} = 0$ when
$s+t > p-1$. Set $c = j(1) - j(0)$.

For $\beta, \alpha \in R$ and $s \ge 0$ put $\nu(\alpha,\beta,s) = \sum_{i=0}^{s-1}
\beta^{i}\alpha^{s-1-i}$, so $\beta^{s} - \alpha^{s} = (\beta-\alpha)\,\nu(\alpha,\beta,s)$ and
$\nu(\alpha,\beta,0) = 0$. For $a,b,c',d' \in R$, expanding each monomial gives the mixed
second difference
\[
  q(b,d') - q(a,d') - q(b,c') + q(a,c')
  = \sum_{s,t} c_{s,t}\,(b^{s}-a^{s})(d'^{\,t}-c'^{\,t}).
\]
Take $a = c' = j(0)$, $b = j(u)$, $d' = j(v)$ for units $u,v \in \{1,\dots,p-1\}$ of
$\mathbb{Z}/p$. The factors with $s = 0$ or $t = 0$ vanish, and $c_{s,t} = 0$ unless
$s+t \le p-1$, so $1 \le s,t \le p-2$; using the factorisation,
\begin{equation}\label{eq:mix}
  q(j u, j v) - q(j 0, j v) - q(j u, j 0) + q(j 0, j 0)
  = \bigl(j u - j 0\bigr)\bigl(j v - j 0\bigr)
    \sum_{s=1}^{p-2}\sum_{t=1}^{p-2} c_{s,t}\,\nu(j0, ju, s)\,\nu(j0, jv, t). \tag{$\dagger$}
\end{equation}
On the other hand $\iota_p(0,v) = \iota_p(u,0) = \iota_p(0,0) = 0$ and $\iota_p(u,v) =
[\,uv \equiv 1\,]$, so by hypothesis the left side of \eqref{eq:mix} equals
$j(\iota_p(u,v)) - j(0) = c\,[\,uv\equiv 1\,]$.

Index rows and columns of matrices over $R$ by the units $1,\dots,p-1$. Define
\[
  D = \operatorname{diag}\bigl(j u - j 0\bigr), \quad
  \Phi_{u,s} = \nu(j0, ju, s)\ (1 \le s \le p-2), \quad
  A_{s,t} = c_{s,t}, \quad
  F_{u,v} = [\,uv \equiv 1 \bmod p\,].
\]
Then \eqref{eq:mix} and the value just computed say exactly
\[
  D\,(\Phi A \Phi^{\mathsf T})\,D = c\,F ,
\]
an identity of $(p-1)\times(p-1)$ matrices. Here $\Phi$ is $(p-1)\times(p-2)$ and
$A\Phi^{\mathsf T}$ is $(p-2)\times(p-1)$, so by \Cref{lem:rank},
$\det(\Phi A \Phi^{\mathsf T}) = \det\bigl(\Phi\,(A\Phi^{\mathsf T})\bigr) = 0$; hence
$\det\bigl(D(\Phi A\Phi^{\mathsf T})D\bigr) = (\det D)^2 \det(\Phi A\Phi^{\mathsf T}) = 0$.
Since $u \mapsto u^{-1}$ is a permutation of $(\mathbb{Z}/p)^\times$, $F$ is its permutation
matrix and $\det F = \pm 1$, a unit. Taking determinants in $D(\Phi A\Phi^{\mathsf T})D = cF$
gives $c^{\,p-1}\det F = 0$, so $c^{\,p-1} = 0$ and $c$ is nilpotent.

Finally take $u = v = 1$. Then $uv = 1$, $j 1 - j 0 = c$, and the $(1,1)$ entry of the
identity reads $c^2 e = c$, where $e = \sum_{s,t} c_{s,t}\,\nu(j0,j1,s)\,\nu(j0,j1,t) \in R$.
By \Cref{lem:core}, $c = 0$, i.e.\ $j(1) = j(0)$; as $1 \ne 0$ in $\mathbb{Z}/p$ this
contradicts injectivity of $j$.
\end{proof}

\begin{theorem}[Non-embeddability of the inversion indicator]\label{thm:obstruction}
\lean{Finbin.not_embedsInDegree_invIndicator}\leanok\uses{def:embed,def:inv,thm:invp}
For every prime $p$, the inversion indicator $\iota_p$ does not embed in degree $p-1$.
\end{theorem}

\begin{proof}\leanok\uses{thm:invp}
An embedding in degree $p-1$ supplies a ring $R$, an injection $j$, and $q$ with
$\deg q \le p-1$ and $q(j(x),j(y)) = j(\iota_p(x,y))$ for all $x,y$, contradicting
\Cref{thm:invp}.
\end{proof}

\begin{corollary}[Arbitrary degree is required]\label{cor:arbitrary}
\lean{Finbin.arbitrary_degree_obstruction}\leanok\uses{thm:obstruction}
For every $d$ there is a binary function over a finite ring that embeds in no polynomial of
total degree $d$. Concretely, for any prime $p > d$ the inversion indicator $\iota_p$ is
such a function.
\end{corollary}

\begin{proof}\leanok\uses{thm:obstruction,lem:mono}
Choose a prime $p \ge d+1$ (the primes are unbounded). Then $d \le p-1$, so by
\Cref{lem:mono} an embedding of $\iota_p$ in degree $d$ gives one in degree $p-1$,
contradicting \Cref{thm:obstruction}.
\end{proof}

\section{Conclusion}

The constructions of \Cref{sec:unary,sec:quartic} are over $\mathbb{Z}/n$; the lower bound
of \Cref{sec:obstruction} holds over every commutative ring. Two questions are left open:
the minimal degree in which a given binary function embeds, and the corresponding statements
for arity $k > 2$.

\end{document}